\documentclass[11pt]{article}
\usepackage{amsmath}
\usepackage{amssymb}
\usepackage{amsthm}
\usepackage[mathscr]{eucal}
\theoremstyle{definition}
\newtheorem{definition}{Definition}
\newtheorem{theorem}[definition]{Theorem}
\newtheorem{proposition}[definition]{Proposition}
\newtheorem{lemma}[definition]{Lemma}
\newtheorem{corollary}[definition]{Corollary}
\theoremstyle{remark}
\newtheorem{remark}[definition]{Remark}

\newcounter{enumctr}

\newcommand{\R}{\mathbb{R}}

\newcommand{\C}{\mathbb{C}}

\renewcommand{\phi}{\varphi}

\newcommand{\rT}{\mathrm {T}}

\begin{document}
\title{\vspace*{-10mm}
Generation of nonlocal fractional dynamical systems by fractional differential equations}
\author{
N.D.~Cong\footnote{\tt ndcong@math.ac.vn, \rm Institute of Mathematics, Vietnam Academy of Science and Technology, 18 Hoang Quoc Viet, 10307 Ha Noi, Viet Nam},
%
%
\;and\;
H.T.~Tuan\footnote{\tt httuan@math.ac.vn, \rm Institute of Mathematics, Vietnam Academy of Science and Technology, 18 Hoang Quoc Viet, 10307 Ha Noi, Viet Nam}
}
\date{}
\maketitle
\begin{abstract}
We show that any two trajectories of solutions of a one-dimensional fractional differential equation (FDE) either coincide or do not intersect each other. In contrary, in the higher dimensional case, two different trajectories can meet. Furthermore, one-dimensional FDEs and triangular systems of FDEs generate nonlocal fractional dynamical systems, whereas a higher dimensional FDE does, in general, not generate a nonlocal dynamical system. 
\end{abstract}
\section{Introduction}
In recent years, fractional differential equations (FDEs) have attracted increasing interest due to
the fact that many mathematical problems in science and engineering can be modeled by fractional differential equations, see e.g.,\ \cite{Podlubny1999, Kilbas2006, Diethelm2010}. 

In this paper, we consider a $d$-dimensional fractional differential equation involving the \emph{Caputo derivative of order $\alpha \in (0,1)$} on an finite interval $J:= [0,T] $ or on the half real line $J:=[0,\infty)$:
\begin{equation}\label{mainEq}
^{C\!}D_{0+}^\alpha x(t)=f(t,x(t)),
\end{equation}
where $x: J \rightarrow \R^d$ is a vector function,  $f: J\times \R^d\rightarrow \R^d$ is a continuous vector-valued function. A continuous function $x: J \rightarrow \R^d$ is called a {\em solution} of \eqref{mainEq} if it satisfied \eqref{mainEq} for all $0<t\in J$. The value $x(0)$  is called the initial value of the solution $x(\cdot)$.

We are interested in the dynamical properties of \eqref{mainEq}: we want to know whether \eqref{mainEq} generates a dynamical system so that the tools and methods of the classical theory of dynamical systems are applicable in the investigation of FDEs. Another important problem in the theory of FDEs, which is closely related to the problem of generation of dynamical systems by FDEs is the question of 
 whether two different trajectories of a FDE can intersect.
We solved both these problems. Namely we show that  
one-dimensional FDEs and triangular FDEs generate nonlocal fractional dynamical systems, whereas a higher dimensional FDE does, in general, not generate a nonlocal dynamical system. Correspondingly, two different trajectories of a one-dimensional or a triangular FDE cannot meet, whereas different trajectories of a high dimensional FDE may intersect each other. As a byproduct of our investigation we get  lower bounds for the solutions of one-dimensional FDEs and triangular linear FDEs.

The question of whether solutions of \eqref{mainEq} can intersect are treated in Diethelm~\cite{Diethelm2008,Diethelm2010}, Diethelm and Ford~\cite{DiethelmFord2012}, Agarwal \textit{et al.}~\cite{Agarwal2010}, Hayek \textit{et al.}~\cite{Hayek1999}, and Bonilla, Rivero and Trujillo~\cite{Bonilla2007}. Note that in the case of ordinary differential equations (ODEs) it is well known that two trajectories of an ODE either coincide or they do not intersect; the authors mentioned above proved that similar results hold for fractional differential equations of order $0<\alpha<1$. Note that the main difficulty of the problem for FDEs is the nonlocal nature (or history memory) of solutions of FDEs. The above authors used various tools to deal with the FDE case. However we notice flaws in their proofs which make their proofs incomplete. We will present some discussion about this matter in Section~\ref{sec.separation1d} of the paper.

The paper is organized as follows: 
Section \ref{sec.preliminaries} is a preparatory section where we present some basic notions from fractional calculus and the theory of  FDEs.
Section \ref{sec.separation1d} is devoted to results on separation of solutions of one-dimensional FDEs; in this section we also give discussion about flaws in the proofs of separation of solutions of FDEs in the above mentioned papers.
In Section \ref{sec.1dimNonlocalDS} we present results on generation of nonlocal fractional dynamical systems by one-dimensional FDEs.
Section~\ref{sec.triangular} is devoted to high dimensional triangular systems of FDEs, where based on the results in Section~\ref{sec.1dimNonlocalDS} we show that a triangular system of FDEs does generate a nonlocal fractional dynamical system.
 In Section~\ref{sec.d-dim} we show that a higher dimensional FDE does, in general, not generate a nonlocal dynamical system; two different trajectories of a high dimensional FDE may intersect each other in finite time.

\section{Preliminaries}\label{sec.preliminaries}

We start this section by briefly recalling a framework of fractional calculus and fractional differential equations. We refer the reader to the books \cite{Diethelm2010,Kilbas2006} for more details.

Let $\R^d$ be the standard $d$-dimensional Euclidean space equipped with usual Euclidean norm.
We denote by $\R_+$ the set of all nonnegative real numbers.
Denote by $C([0,\infty);\R^d)$ the space of continuous functions from $[0,\infty)$ to $\R^d$,
 and by $\left(C_\infty(\R^d),\|\cdot\|_\infty\right)\subset C([0,\infty);\R^d)$ the space of all continuous functions $\xi:\R_+\rightarrow \R^d$ 
which are uniformly bounded on $\R_+$, i.e.\
\[
\|\xi\|_\infty:=\sup_{t\in \R_+}\|\xi(t)\|<\infty.
\]
It is well known that $\left(C_\infty(\R^d),\|\cdot\|_\infty\right)$ is a Banach space. 

Let $\alpha>0$ and $[a,b]\subset \R$. Let $x:[a,b]\rightarrow \R$ be a measurable function such that $x\in L^1([a,b])$, i.e.,\ $\int_a^b|x(\tau)|\;d\tau<\infty$. Then, the Riemann-Liouville integral of order $\alpha$ is defined by
\[
I_{a+}^{\alpha}x(t):=\frac{1}{\Gamma(\alpha)}\int_a^t(t-\tau)^{\alpha-1}x(\tau)\;d\tau\quad \hbox{ for } t\in (a,b],
\]
where the Gamma function $\Gamma:(0,\infty)\rightarrow \R$ is defined as
\[
\Gamma(\alpha):=\int_0^\infty \tau^{\alpha-1}\exp(-\tau)\;d\tau,
\]
see e.g., Diethelm~\cite{Diethelm2010}. The corresponding Riemann--Liouville fractional derivative of order $\alpha $ is given by
\[
^{R\!}D_{a+}^\alpha x(t):=(D^m I_{a+}^{m-\alpha}x)(t) \qquad\hbox{ for } t\in (a,b],
\]
where $D=\dfrac{d}{dt}$ is the usual derivative and $m:=\lceil\alpha\rceil$ is the smallest integer bigger or equal to $\alpha$. On the other hand, the \emph{Caputo fractional derivative } $^{C\!}D_{a+}^\alpha x$ of a function $x\in C^m([a,b])$, which was introduced by Caputo (see e.g., Diethelm~\cite{Diethelm2010}), is defined by
\[
^{C\!}D_{a+}^\alpha x(t):=(I_{a+}^{m-\alpha}D^m x)(t),\qquad \hbox{ for } t\in (a,b].
\]
The Caputo fractional derivative of a $d$--dimensional vector function $x(t)=(x_1(t),\dots,x_d(t))^{\rT}$ is defined component-wise as $$^{C\!}D^\alpha_{0+}x(t)=(^{C\!}D^\alpha_{0+}x_1(t),\dots,^{C\!}D^\alpha_{0+}x_d(t))^{\rT}.$$

It is well-known that the initial value problem of the FDE \eqref{mainEq} is equivalent to a Volterra integral equation of the second kind. Namely we have the following 

\begin{lemma}\label{lem.integral.form}
A continuous function $x: J\rightarrow \R$ is a solution of the FDE \eqref{mainEq} with the initial value condition $x(0)=x_0$ if and only if it is a solution of the Volterra integral equation of the second kind
\begin{equation}\label{eqn.volterra}
x(t) = x_0 +\frac{1}{\Gamma(\alpha)}\int_0^t (t-\tau)^{\alpha-1}f(\tau,x(\tau)) d\tau.
\end{equation}
\end{lemma}
For a proof we refer to Diethelm~\cite[Lemma 6.2, p. 86]{Diethelm2010} and Kilbas \textit{et al.}~\cite[Theorem 3.24, p. 199]{Kilbas2006} for the one-dimensional case; the multidimensional case is just componentwise application of the one-dimensional case.

\section{Separation of trajectories of solutions of one-dimensional FDE}\label{sec.separation1d}

\subsection{Two different trajectories of a one-dimensional FDE do not meet}
In this section, we consider the one-dimensional case of the system \eqref{mainEq}, i.e. the following FDE
\begin{equation}\label{1dimEq}
^{C\!}D_{0+}^{\alpha}x(t)=f(t,x(t)),
\end{equation}
where $f:J\times \R \to \R$ is a continuous function. Assume that $f$ satisfies the following Lipschitz condition on the second variable: there exists a nonnegative continuous function $L: J \rightarrow \R_+$ such that
\begin{equation}\label{eqn.Lip}
|f(t,x)-f(t,y)| \leq L(t)|x-y|\qquad \hbox{for all}\quad t\in J\quad\hbox{and all}\quad x,y\in\R.
\end{equation}
It is well known that under the Lipschitz condition \eqref{eqn.Lip} the initial value problem for \eqref{1dimEq} has unique solution defined on the whole interval $J$ for any given initial value (see, e.g. Baleanu and Mustafa~\cite[Theorem 2]{Baleanu2010}, Tisdell~\cite[Theorem 6.4]{Tisdell2012}) and Diethelm~\cite[Theorem 6.8]{Diethelm2010}).
We will show that any two solutions of \eqref{1dimEq} either coincide or do not intersect on $J$. For this we need the variation of constant formula for FDEs and the comparision principle, which are useful technical tools for investigation of FDEs.  

\begin{lemma}[Variation of constants formula for FDEs]\label{lem.variationofconstant}
If the function $f$ in the FDE \eqref{1dimEq}
is of the form
$$
f(t,x) = Mx + g(t,x)
$$
for some fixed $M\in \R$ and all $t\in J$, $x\in \R$, then the solution $x(\cdot)$ of \eqref{1dimEq} with the initial value $x(0)=x_0$ satisfies for all $t\in J$ the formula
$$
x(t) = E_\alpha(Mt^\alpha)x_0 + \int_0^t (t-\tau)^{\alpha-1} E_{\alpha,\alpha}(M(t-\tau)^{\alpha})g(\tau,x(\tau)) d\tau,
$$
where 
$$
E_{\alpha,\beta}(z):=\sum_{k=0}^\infty \frac{z^k}{\Gamma(\alpha k+\beta)},\qquad E_{\alpha}(z):=E_{\alpha,1}(z),\qquad z\in\C,
$$
are Mittag-Leffler functions and $\Gamma(\cdot)$ is the Gamma function.
\end{lemma}
For a proof of this theorem we refer the reader to Kilbas \textit{et al.}~\cite[Theorem 5.15, p.~323]{Kilbas2006} and Diethelm~\cite[Theorem 7.2, p.~135]{Diethelm2010}.

\begin{theorem}[Comparision Principle~\cite{Lakshmikantham2008}]\label{comparision}
Let $0<q<1$,  $v,w\in C(J,\R)$ are continuous function on $J$, $g\in C(J\times \R,\R)$ is a continuous two-variable funtion on $J\times \R$, and \\
(i) $v(t) \leq v(0) + \dfrac{1}{\Gamma(q)}\int_0^t (t-s)^{q-1}g(s,v(s)) ds$ for all  $t\in J$,\\
(ii) $w(t) \geq w(0) + \dfrac{1}{\Gamma(q)}\int_0^t (t-s)^{q-1}g(s,w(s)) ds$ for all $t\in J$.\\
Suppose further that $g(t,x)$ is nondeacreasing in $x$ for each $t\in J$, and
\begin{equation}\label{eqn.compar1}
v(0) < w(0).
\end{equation}
Then we have
\begin{equation}\label{eqn.compar2}
v(t) < w(t) \qquad\hbox{for all}\quad t\in J.
\end{equation}
\end{theorem}
For convenience of the reader we present here the proof of this theorem by 
 Lakshmikantham and Vatsala~\cite[Theorem 2.1]{Lakshmikantham2008} with a small modification.
\begin{proof}
Suppose that \eqref{eqn.compar2} is not true. Then, because of the continuity of $v(\cdot)$ and $w(\cdot)$
there exists $0< t_1\in J$ such that
\begin{equation}\label{eqn.compar3}
v(t_1) = w(t_1), \qquad\hbox{and}\qquad v(t) <w(t)\quad\hbox{for all}\quad  0\leq t <t_1.
\end{equation}
Using the nondecreasing nature of $g$ and \eqref{eqn.compar3}, taking into account  \eqref{eqn.compar1}, we get
\begin{eqnarray*}
v(t_1) &\leq& v(0) + \dfrac{1}{\Gamma(q)}\int_0^{t_1} (t_1-s)^{q-1}g(s,v(s)) ds\\
&\leq& v(0) + \dfrac{1}{\Gamma(q)}\int_0^{t_1} (t_1-s)^{q-1}g(s,w(s)) ds\\
&<& w(0) + \dfrac{1}{\Gamma(q)}\int_0^{t_1} (t_1-s)^{q-1}g(s,w(s)) ds\\
&\leq& w(t_1),
\end{eqnarray*}
which contradict the condition $v(t_1) = w(t_1)$ in \eqref{eqn.compar3}. Hence \eqref{eqn.compar2} is valid and the proof is complete.
\end{proof}

Now we are in a position to prove our main theorem of this section.
\begin{theorem}[Different trajectories do not meet]\label{separation1dim}
 Assume that $f$ satisfies the Lipschitz condition \eqref{eqn.Lip}. Then for any two different initial values $x_{10}\not= x_{20}$ in $\R$ the trajectories of solutions of the FDE \eqref{1dimEq} do not meet on $J$, i.e., the solutions $x_1(\cdot),x_2(\cdot)$ of 
\eqref{1dimEq} starting from $x_{10}=x_1(0)$ and $x_{20}=x_2(0)$ verify $x_1(t)\not= x_2(t)$ for all $t\in J$.
\end{theorem}
\begin{proof}
For definiteness we assume that 
\begin{equation}\label{eqn.1dim1}
x_1(0)=x_{10} < x_{20}=x_2(0).
\end{equation}
To prove the theorem we show that $x_1(t)< x_2(t)$ for all $t\in J$.
Suppose that this is not true, then there is $0<T_1\in J$ such that $x_1(T_1)\geq x_2(T_1)$. By the continuity property of $x_1(\cdot), x_2(\cdot)$ and \eqref{eqn.1dim1} there is $T_2>0$ such that
\begin{eqnarray}
x_1(T_2)&=& x_2(T_2), \quad\hbox{and} \label{eqn.1dim2}\\
x_1(t) &<& x_2(t)\qquad\hbox{for all}\quad 0\leq t < T_2. \label{eqn.1dim3}
\end{eqnarray}
Put
\begin{eqnarray}
M &:=& \max_{0\leq t\leq T_2} L(t),  \label{eqn.1dim4}\\
g(t,x) &:=& f(t,x) + Mx\qquad \hbox{for all}\quad 0\leq t\leq T_2, x\in\R.\label{eqn.1dim5}
\end{eqnarray}
Then $g: [0,T_2]\times \R \rightarrow \R$ is nondecreasing in $x$ for each $t\in [0,T_2]$ since by \eqref{eqn.Lip}, \eqref{eqn.1dim4} and \eqref{eqn.1dim5}, for any $t\in [0,T_2]$ and $x\leq y$ we have
$$
g(t,y)-g(t,x) = M(y-x) +f(t,y)-f(t,x)\geq (M-L(t))(y-x)\geq 0.
$$
By virtue of Lemma~\ref{lem.variationofconstant}, since $x_1(\cdot)$ and $x_2(\cdot)$ are solutions of \eqref{1dimEq}, on the interval $[0,T_2]$ we have the following equations
\begin{eqnarray}
x_1(t) &=& E_\alpha(-Mt^\alpha)x_{10} \nonumber\\
&&\hspace*{0.1cm}+\;\int_0^t (t-\tau)^{\alpha-1}E_{\alpha,\alpha}(-M(t-\tau)^\alpha)g(\tau,x_1(\tau)) d\tau,\label{eqn.1dim6}\\[6pt]
x_2(t) &=& E_\alpha(-Mt^\alpha)x_{20} \nonumber\\
&&\hspace*{0.1cm}+\;\int_0^t (t-\tau)^{\alpha-1}E_{\alpha,\alpha}(-M(t-\tau)^\alpha)g(\tau,x_2(\tau)) d\tau.\label{eqn.1dim7}
\end{eqnarray}
Since $E_{\alpha,\alpha}(s)>0$ for all $s\in\R$ (see, e.g., Cong \textit{et al.}~\cite[Lemma 2]{Cong}), the function 
$E_{\alpha,\alpha}(-M(t-\tau)^\alpha)g(\tau,x)$
is nondecreasing in the variable $x$ for all $0\leq\tau\leq t\leq T_2$. Therefore, a small modification of 
Theorem~\ref{comparision} is applicable to the pair of integral equations \eqref{eqn.1dim6}-\eqref{eqn.1dim7} on $[0,T_2]$ and gives $x_1(T_2) < x_2(T_2)$. Thus we arrive at a contradiction. Consequently, conclusion of the theorem is true and the proof is completed.
\end{proof}
 
 \begin{remark}
 Theorem~\ref{separation1dim} provides a full solution to Conjecture 1.2 of 
 Diethelm~\cite{Diethelm2008}. Note that in \cite{Diethelm2008} Diethelm provides a partial solution to the problem of separation of trajectories: he proved Conjecture 1.2 in \cite{Diethelm2008} under some restrictive conditions, see \cite[Theorem 2.2]{Diethelm2008}.
 \end{remark}

\subsection{Discussion on the problem of separation of solutions of a general FDE}\label{subsec.discussion}

Diethelm~\cite[Theorem 6.12]{Diethelm2010}
 formulated and proved a theorem on separation of solutions of a FDE, what is the same as our Theorem~\ref{separation1dim}. For the proof he used a fixed point theorem on a short interval of time, then used induction to extend the result to the whole (long) interval of time under consideration. However, his proof contains a flaw in the induction part: the passage to the next step from $N=1$ to $N=2$ does not work since the FDE is history dependent and on the next subinterval the argument leading to contractive mapping do not work like on the first subinterval of time. 

Diethelm, in his joint work with Ford \cite[Theorems 3.1 and 4.1]{DiethelmFord2012}, gave and alternative (to \cite[Theorem 6.12]{Diethelm2010}) proof of the theorem on separation of solution of FDE. Namely, instead of induction forward in time from one subinterval to the next as in \cite{Diethelm2010} they used induction backward in time from one subinterval to the foregoing interval. However, their argument in the first step of induction does not work:
in the notation of \cite{DiethelmFord2012}, the equation (13) is equivalent to equation (8) only if we consider (8) and (13) in the whole interval $[0,b]$; actually the solution $y(t)$, considered as solution of the terminal problem (13), depends on the future value $y(b)=c$, hence we cannot say that the function $g(t)$ on pages 29--30 of \cite{DiethelmFord2012} are independent of the restriction of the function $y$ to $[T_{N-1},T_N]$, therefore the contraction property of the map $\mathcal F$  in the formula (14) of \cite{DiethelmFord2012} does not lead to existence and uniqueness of solution of (14) as claimed in \cite{DiethelmFord2012}. Therefore, the proof of Theorem 3.1 in \cite{DiethelmFord2012} is incomplete, hence so is Theorem 4.1 of \cite{DiethelmFord2012}.

Note that the proof of Theorem 6.12 of Diethelm~\cite{Diethelm2010}  is correct for a first "short" interval of time where the smallness of time combined with bounded Lipschitz condition of $f$ make certain operator contractive, hence two different solutions can not meet. 
By the way, we notice that the continuity alone can also assure the non-intersection of different trajectories in a short time interval. 

Hayek, Trujillo, Rivero, Bonilla and 
Boreno \cite[Theorem 3.1]{Hayek1999} have proved the separation theorem on a "short" interval of time. We also note that Section 4 of \cite{Hayek1999} is false because of the history dependence of solutions of FDE, which prevents us from usage of "usual method of prolongation" (like in ODE) as claimed by the authors of \cite{Hayek1999}. 

Bonilla, Rivero and Trujillo~\cite{Bonilla2007} treated higher dimensional linear systems of FDEs and in Section 3 of \cite{Bonilla2007} they relied on the above mention result by Hayek \textit{et al.}~\cite{Hayek1999}, hence there are gaps in their proofs of some results, namely the proofs of Theorem 1 on page 71 and of Propositions 1 and 2 on page 72 of \cite{Bonilla2007} are incomplete. For a counter example see Section~\ref{sec.d-dim}.

\section{One-dimensional FDEs generate nonlocal dynamical systems}\label{sec.1dimNonlocalDS}

In this section, based on the results on separation of trajectories presented in Section~\ref{sec.separation1d} we show that one-dimensional FDEs generate nonlocal dynamical systems, hence tools and methods of the classical theory of dynamical systems can be applied. 

\subsection{Bounds for solutions of FDEs}
First we formulate and prove a lower bound for solutions of FDE, which provides us with a better understanding of geometry of solutions of one-dimensional fractional differential equations.
\begin{theorem}[Convergence rate for solutions of 1-dim FDEs]\label{thm.convergence.rate}
Assume that $f$ satisfies the Lipschitz condition \eqref{eqn.Lip}.
For any two solutions $x_1(\cdot), x_2(\cdot)$ of the FDE \eqref{1dimEq} and any $t\in J$ the following estimate holds
$$
|x_2(t)-x_1(t)| \geq |x_2(0) - x_1(0)| E_\alpha\big(- (\max_{0\leq \tau\leq t} L(\tau)) t^\alpha\big).
$$
\end{theorem}
\begin{proof}
 For definiteness we assume $x_2(0) \geq x_1(0)$. Then by Theorem \ref{separation1dim}, we have $x_2(t)\geq x_1(t)$ for any $t\in J$. For an arbitrary $t\geq 0$ but fixed, put
$M_t:=\sup_{0\leq s\leq t}L(s)$. On the interval $[0,t]$, repeat the arguments of the proof of Theorem~\ref{separation1dim}, we get
\begin{eqnarray*}
x_2(s) -x_1(s) &=& E_\alpha(-M_ts^\alpha)(x_2(0)-x_1(0)) +\\
&&\hspace*{-2.1cm}+\;\int_0^ s (s-\tau)^{\alpha-1}E_{\alpha,\alpha}(-M_t(s-\tau)^\alpha)( g(\tau,x_2(\tau))- g(\tau,x_1(\tau))) d\tau\\[6pt]
&\geq&\hspace*{0.1cm}  E_\alpha(-M_ts^\alpha)(x_2(0)-x_1(0)).
\end{eqnarray*}
Take $s=t$, and the proof is complete.
\end{proof}

\begin{corollary}[Lower bound for solutions of 1-dim FDEs]\label{cor.lower.bound}
Assume that $f$ satisfies the Lipschitz condition \eqref{eqn.Lip}. Assume additionally that $f(t,0)=0$ for all $t\in J$, $x\in\R$. Then for
 any solution $x(\cdot)$ of the FDE \eqref{1dimEq} and any $t\in J$ the following estimate holds
$$
|x(t)| \geq |x(0)| E_\alpha\big(- (\max_{0\leq \tau\leq t} L(\tau)) t^\alpha\big).
$$
\end{corollary}
\begin{proof}
Since $f(t,0)=0$ the FDE \eqref{1dimEq} has trivial solution. Apply Theorem~\ref{thm.convergence.rate} for the pair of $x(\cdot)$ and the trivial solution of \eqref{1dimEq}.
\end{proof}

For the divergence rate and upper bound for solutions of the FDEs the following statements are easy modifications of well known results.

\begin{theorem}[Divergence rate for solutions of 1-dim FDEs]\label{thm.divergence.rate}
Assume that $f$ satisfies the Lipschitz condition \eqref{eqn.Lip}.
For any two solutions $x_1(\cdot), x_2(\cdot)$ of the FDE \eqref{1dimEq} and for any $t\in J$ the following estimate holds
$$
|x_2(t)-x_1(t)| \leq |x_2(0) - x_1(0)| E_\alpha\big( (\max_{0\leq \tau\leq t} L(\tau)) t^\alpha\big).
$$
\end{theorem}
\begin{proof}
Let us denote $M_t = \max_{0\leq \tau\leq t} L(\tau)$. By Lemma~\ref{lem.integral.form}, for all $t\in J$ we have
\begin{eqnarray*}
x_1(t) &=& x_1(0) +\frac{1}{\Gamma(\alpha)}\int_0^t (t-\tau)^{\alpha-1}f(\tau,x_1(\tau)) d\tau,\\
x_2(t) &=& x_2(0) +\frac{1}{\Gamma(\alpha)}\int_0^t (t-\tau)^{\alpha-1}f(\tau,x_2(\tau)) d\tau.
\end{eqnarray*}
Therefore, for any $t\in J$ and  any $0\leq s\leq t$ we have
\begin{eqnarray*}
|x_2(s)-x_1(s)| &\leq& |x_2(0)-x_1(0)| + \\
&&\hspace*{0.5cm} + \frac{1}{\Gamma(\alpha)}\int_0^s (s-\tau)^{\alpha-1}(f(\tau,x_2(\tau)) - f(\tau,x_1(\tau)) ) d\tau\\[5pt]
&\leq& |x_2(0)-x_1(0)| + \frac{1}{\Gamma(\alpha)}\int_0^s (s-\tau)^{\alpha-1} M_t|x_2(\tau)-x_1(\tau)| d\tau.
\end{eqnarray*}
By virtue of the Gronwall-type inequality for FDE (see Diethelm~\cite[Theorem 6.19, p. 111]{Diethelm2010} and Tisdell~\cite[Lemma 3.1, p. 288]{Tisdell2012}), this implies
$$
|x_2(s)-x_1(s)| \leq |x_2(0) - x_1(0)| E_\alpha\big( M_t s^\alpha\big).
$$
Put $s=t$ and since $t\in J$ is arbitrary we have the theorem proved.
\end{proof}

\begin{corollary}[Upper bound for solutions of 1-dim FDEs]\label{cor.upper.bound}
 Assume that $f$ satisfies the Lipschitz condition \eqref{eqn.Lip}. Assume additionally that $f(t,0)=0$ for all $t\in J$, $x\in\R$. Then for
 any solution $x(\cdot)$ of the FDE \eqref{1dimEq} and any $t\in J$ the following estimate holds
$$
|x(t)| \leq |x(0)| E_\alpha\big((\max_{0\leq \tau\leq t} L(\tau)) t^\alpha\big).
$$
\end{corollary}
\begin{proof}
Since $f(t,0)=0$ the FDE \eqref{1dimEq} has trivial solution. Apply Theorem~\ref{thm.divergence.rate} for the pair of $x(\cdot)$ and the trivial solution of \eqref{1dimEq}.
\end{proof}

\begin{remark}\label{rem.multidim}
It is easily seen that Theorem~\ref{thm.divergence.rate} and Corollary~\ref{cor.upper.bound} hold true also for the case of a high dimensional system of FDEs. 
\end{remark}

\subsection{One-dimensional FDEs generate two-parameter flows}\label{subsec.1dim} 

Now we are in a position to show that one-dimensional FDEs generate two-parameter flows. First we define the evolution mappings of  \eqref{1dimEq}.
\begin{definition}\label{dfn.0evolution-mapping}
The mapping
\begin{equation}
\Phi_{0,T_1} :\R\rightarrow \R, \quad x_0 \mapsto x(T_1),
\end{equation}
where $x_0\in\R$ is an arbitrary initial value of \eqref{1dimEq}, $x(\cdot)$ is the solution of \eqref{1dimEq} starting from $x(0)=x_0$ and $x(T_1)$ is the evaluation of $x(\cdot)$ at $T_1$, is called {\em the evolution mapping of \eqref{1dimEq}}.
\end{definition}

\begin{definition}\label{2parameterflow}
A two-parameter family of mappings
$$
\phi_{s,t}(\cdot) : \R\rightarrow \R, \qquad s,t\in J,
$$
 is called  {\em a two-parameter flow in $\R$} if $\phi_{s,t}(x)$ is continuous as a function of three variables $s,t\in J, x\in\R$, for any fixed $s,t\in J$ the mapping $\phi_{s,t}$ is a homeomorphism of $\R$, and  this family satisfies the following flow property
$$
\phi_{s,t}\circ \phi_{u,s} = \phi_{u,t}\qquad\hbox{for all}\quad u,s,t\in J.
$$
\end{definition} 

\begin{theorem}[One-dimensional FDEs generate two-parameter flows in $\R$]\label{thm.2param.flow}
The following statements hold for the one-dimensional FDE \eqref{1dimEq}.
\begin{enumerate}
\renewcommand{\labelenumi}{(\roman{enumi})}
\item The evolution mapping $\Phi_{0,t}$ generated by \eqref{1dimEq} is a bijection for any $t\in J$.
\item The FDE \eqref{1dimEq} generates a two-parameter family of bijections on $J$ by its evolution mappings as follows
\begin{equation}\label{eqn.2flow}
\Phi_{s,t} := \Phi_{0,t}\circ \Phi^{-1}_{0,s}\qquad\hbox{for all}\quad s,t\in J,
\end{equation}
where $\Phi_{0,\cdot}$ is the evolution mapping of \eqref{1dimEq} defined in Definition~\ref{dfn.0evolution-mapping}.
\item The family $\Phi_{s,t}$, $s,t\in J$, generated by the FDE \eqref{1dimEq} is a two-parameter flow in $\R$.
\item If $f$ is linear in $x$ then the two-parameter flow generated by the FDE \eqref{1dimEq} is a flow of linear operators.
\end{enumerate}
\end{theorem}
\begin{proof}
(i) Fix $T_1\in J$ arbitrary. By Theorem~\ref{separation1dim}, $\Phi_{0,T_1}$ is injective.

To show that 
$\Phi_{0,T_1}$ is surjective, it suffices to show that for an arbitrary $x^*\in\R$ the boundary value problem
\begin{eqnarray}
&^{C\!}D_{0+}^{\alpha}x(t)=f(t,x(t)),\label{eqn.surj1}\\
&x(T_1) = x^*,\label{eqn.surj2}
\end{eqnarray}
where $t\in [0,T_1]$, $f$ is continuous and satisfies the Lipschitz condition \eqref{eqn.Lip}, has a continuous solution. Put
$$
M_1 := \max_{0\leq t\leq T_1} L(t),
$$
where $L(t)$ is determined from \eqref{eqn.Lip}. Let us denote by $\hat x(\cdot)$ the solution of the FDE \eqref{eqn.surj1} satisfying the initial value condition $\hat x(0)=0$. Put
$$
M_2 :=  \max_{0\leq t\leq T_1} |\hat x(t)|, \qquad M_3:= |x^*-\hat{x}(T_1)| +M_2,
$$
and
$$
M_4 := \frac{M_3}{E_\alpha(-M_1T_1^\alpha)}, \qquad M_5 := M_2 + M_4 E_\alpha(M_1T_1^\alpha)+1.
$$
Clearly $M_2<M_3<M_4<M_5$. Define a function on $[0,T_1]\times \R$ as follows
\begin{equation}\label{eqn.surj3}
\hat{f}(t,x) = 
\begin{cases} f(t,x),&\quad\hbox{if}\quad |x| \leq M_5\\
f(t,M_5\frac{x}{|x|}),&\quad \hbox{if}\quad |x| > M_5,
\end{cases}
\end{equation}
and consider the boundary problem of the FDE
\begin{eqnarray}
&^{C\!}D_{0+}^{\alpha}x(t)=\hat{f}(t,x(t)),\label{eqn.surj4}\\
&x(T_1) = x^*.\label{eqn.surj5}
\end{eqnarray}
Clearly $\hat{f}$ is Lipschitz continuous with Lipschitz constant $L(t)$ and bounded on $[0,T_1]\times \R$, 
hence the boundary problem \eqref{eqn.surj4}--\eqref{eqn.surj5} has at least one solution, say $x_1(\cdot)$ (see \cite[Theorem 8]{Benchohra2008}).
We show that $x_1(\cdot)$ is the required solution of \eqref{eqn.surj1}--\eqref{eqn.surj2} and we will be done. To this end notice that for all $t\in [0,T_1]$ we have
$$
|\hat{x}(t)| \leq M_2 < M_5 \qquad\hbox{hence}\quad f(t,\hat{x}(t)) = \hat{f}(t,\hat{x}(t)).
$$
Therefore, $\hat{x}(\cdot)$ is the solution of the FDE \eqref{eqn.surj4} satisfying the initial value condition $\hat{x}(0)=0$.
Apply Theorem~\ref{thm.convergence.rate} to the solutions $\hat{x}(\cdot), x_1(\cdot)$ of the FDE \eqref{eqn.surj4} we get for any $t\in [0,T_1]$ the inequality
$$
|\hat{x}(t)-x_1(t)| \geq |\hat{x}(0) - x_1(0)| E_\alpha\big(- (\max_{0\leq \tau\leq t} L(\tau)) t^\alpha\big) \geq |x_1(0)| E_\alpha(-M_1 T_1^\alpha).
$$
Substituting $t=T_1$ we get
$$
|\hat{x}(T_1) - x_1(T_1)| \geq |x_1(0)|E_\alpha(-M_1 T_1^\alpha),
$$
hence 
$$
|x_1(0)| \leq \frac{|\hat{x}(T_1) - x^*|}{E_\alpha(-M_1 T_1^\alpha)}\leq M_4.
$$
Apply Theorem~\ref{thm.divergence.rate} to the solutions $\hat{x}(\cdot), x_1(\cdot)$ of the FDE \eqref{eqn.surj4}  we get  for any $t\in [0,T_1]$ the inequality
$$
|x_1(t)| \leq |\hat{x}(t)| + |x_1(0)| E_\alpha(M_1T_1^\alpha) \leq M_5,\qquad\hbox{hence}\quad f(t,x_1(t)) = \hat{f}(t,x_1(t)).
$$
Therefore $x_1(\cdot)$ is a solution of the FDE  \eqref{eqn.surj1}, and (i) is proved.

(ii) By (i), the evolution mappings of \eqref{1dimEq} are bijective, hence  $\Phi_{s,t}$ is well defined by \eqref{eqn.2flow}. The flow property is easily verified.

(iii) By (ii), the FDE \eqref{1dimEq} generates a two-parameter family of bijections $\Phi_{s,t}$ of $\R$ for all $s,t\in J$.
From Theorems \ref{thm.convergence.rate} and \ref{thm.divergence.rate} it follows that the bijections $\Phi_{s,t}$ are homeomorphisms and $\Phi$ depends continuously on three variables $s,t,x$.

(iv) Obvious.
\end{proof}

\begin{definition}
The two-parameter flow $\Phi_{s,t}$, specified in Theorem~\ref{thm.2param.flow}, generated by the FDE \eqref{1dimEq} is called the {\em nonlocal dynamical system generated by \eqref{1dimEq}}.
\end{definition}

\begin{remark}
Some distinguished features of the two-parameter flow generated by the FDE \eqref{1dimEq}:\\
(i) The flow has history memory. Though the past has impact on the behavior of the solutions, the solutions form a two-parameter flow of homeomorphisms.\\
(ii) The flow is in general $\alpha$-H\"older, but not $C^1$.
\end{remark}

\begin{remark}
Li and Ma~\cite[Theorem 2]{LiMa13} claimed that they constructed a dynamical system from a FDE. However their construction is false. For a counter example, see 
Cong, Son and Tuan~\cite[Remark 12]{Cong_3}.
\end{remark}

\section{Triangular systems of FDEs generate nonlocal dynamical systems}\label{sec.triangular}
In this section, using the results of Section~\ref{sec.1dimNonlocalDS} we show that a higher dimensional triangular system of FDEs also generates a nonlocal dynamical system.

Let us consider a $d$-dimensional triangular system of (not necessarily linear) FDEs
\begin{equation}\label{eqn.triangular.nonlin}
\begin{cases}
^{C\!}D_{0+}^{\alpha}x_1(t)= &f_1(t,x_1(t)),\\
^{C\!}D_{0+}^{\alpha}x_2(t)= &f_2(t,x_1(t),x_2(t)),\\
\cdots&\cdots\\
^{C\!}D_{0+}^{\alpha}x_d(t)= &f_d(t,x_1(t),x_2(t),\cdots,x_d(t)),
\end{cases}
\end{equation}
where $t\in J$, $x(\cdot)=\big( x_1(\cdot),\ldots,x_d(\cdot)\big)^\rT \in \R^d$, $f=(f_1,\ldots,f_d)^\rT$ are Lipschitz functions in $x$ variables, i.e., there exists a continuous function $L: J\rightarrow [0,\infty)$ such that for all $i=1,\ldots, d$ and all $t\in J$ we have
\begin{equation}\label{eqn.Lip.tri.nonlin}
|f_i(t,x_1,\ldots,x_i)- f_i(t,y_1,\ldots,y_i)| \leq L(t) \sqrt{(x_1-y_1)^2 + \cdots (x_i-y_i)^2}.
\end{equation}
This triangular system has a distinguished property that it can be solved successively coordinate-wise and each time we have to solve only a one-dimensional FDE, hence the triangular system inherits many  features of the one-dimensional FDEs. 

\begin{proposition}[Convergence rate for solutions of a triangular system of FDEs]\label{tri.nonlin.lower.bound}
Assume that the Lipschitz condition  \eqref{eqn.Lip.tri.nonlin} is satisfied.  Then for
 any two solutions $x(\cdot), y(\cdot)$ of the triangular FDE \eqref{eqn.triangular.nonlin} and any $t\in J$ the following estimate holds
 $$
\|x(t)-y(t)\| \geq \|x(0) - y(0)\| E_\alpha\big(- (\max_{0\leq \tau\leq t} L(\tau)) t^\alpha\big).
$$
\end{proposition}
\begin{proof}
Let $x(\cdot)=(x_1(\cdot),\ldots,x_d(\cdot))^\rT$ and $y(\cdot)=(y_1(\cdot),\ldots,y_d(\cdot))^\rT$ be arbitrary two solutions of  the triangular FDE \eqref{eqn.triangular.nonlin}.
Consider the first equation in \eqref{eqn.triangular.nonlin}: it is a 1-dimensional equation for the first coordinate. Applying Theorem~\ref{thm.convergence.rate} to this equation we get
$$
|x_1(t)-y_1(t)| \geq |x_1(0) - y_1(0)| E_\alpha\big(- (\max_{0\leq \tau\leq t} L(\tau)) t^\alpha\big).
$$
Since the first coordinate is solvable from the first equation we can substitute it into the second equation of the system \eqref{eqn.triangular.nonlin} and get a 1-dimensional FDE for the second coordinate
$$
^{C\!}D_{0+}^{\alpha} u(t)= f_2(t,x_1(t),u(t)) =: \hat{f}_2(t,u(t)),
$$
where due to \eqref{eqn.Lip.tri.nonlin} the function $\hat{f}_2(\cdot,\cdot): J\times \R \rightarrow \R$ is $L(t)$-Lipschitz continuous with respect to the second variable.
Applying Theorem~\ref{thm.convergence.rate} to the solutions $x_2(\cdot)$ and $y_2(\cdot)$ of  this 1-dimensional FDE we get
$$
|x_2(t)-y_2(t)| \geq |x_2(0) - y_2(0)| E_\alpha\big(- (\max_{0\leq \tau\leq t} L(\tau)) t^\alpha\big).
$$
Continue this process we get for any $i=1,\ldots,d$ and any $t\in J$ the inequality
\begin{equation}\label{eqn.tri.nonlin1}
|x_i(t)-y_i(t)| \geq |x_i(0) - y_i(0)| E_\alpha\big(- (\max_{0\leq \tau\leq t} L(\tau)) t^\alpha\big).
\end{equation}
This implies immediately the conclusion of the proposition.
\end{proof}

An important particular case of the triangular system of  FDEs \eqref{eqn.triangular.nonlin} is a linear triangular system of FDEs
\begin{equation}\label{triangularEq}
^{C\!}D_{0+}^{\alpha}x(t)=A(t)x(t),
\end{equation}
where $t\in J, x\in\R^d$,  $A:J\times \R^{d\times d}$ is a bounded continuous triangular $(d\times d)$-matrix function, i.e., 
$A(\cdot) = \Big( a_{ij}(\cdot)\Big)_{1\leq i,j\leq d}$
where either $a_{ij}=0$ for all $i>j$ (upper triangular) or $a_{ij}=0$ for all $i<j$ (lower triangular), and there exists a continuous function $L: J \rightarrow [0,\infty)$ such that
\begin{equation}\label{eqn.tri1}
\|A(t)\| \leq L(t)\qquad \hbox{for all}\quad t\in J.
\end{equation}
Clearly, Proposition~\ref{tri.nonlin.lower.bound} is applicable to the linear triangular system
\eqref{triangularEq}. Moreover, we also have a lower bound for solutions of \eqref{triangularEq}.

\begin{proposition}[Lower bound for solutions of a linear triangular system of FDEs]\label{tri.lower.bound}
Assume that the triangular matrix function $A$ satisfies the conditions \eqref{eqn.tri1}.  Then for
 any solution $x(\cdot)$ of the FDE \eqref{triangularEq} and for any $t\in J$ the following estimate holds
$$
\|x(t)\| \geq \|x(0)\| E_\alpha(-  (\max_{0\leq \tau\leq t} L(\tau))t^\alpha).
$$
\end{proposition}
\begin{proof}
Since the system \eqref{triangularEq} has trivial solution $0$, we can apply
Proposition~\ref{tri.nonlin.lower.bound} to the two solution $x(\cdot)$ and 0 of \eqref{triangularEq} and get the conclusion of the proposition.
\end{proof}

Similar to the one-dimensional case, for any $T_1\in J$ the mapping
\begin{equation}
\phi_{0,T_1} :\R^d\rightarrow \R^d, \quad x_0 \mapsto x(T_1),
\end{equation}
where $x_0\in\R^d$ is an arbitrary initial value of \eqref{eqn.triangular.nonlin}, $x(\cdot)$ is the solution of \eqref{eqn.triangular.nonlin} starting from $x(0)=x_0$ and $x(T_1)$ is the evaluation of $x(\cdot)$ at $T_1$, is called {\em the evolution mapping of \eqref{eqn.triangular.nonlin}}. By the same arguments as in Section~\ref{sec.1dimNonlocalDS} we can show that the evolution mapping $\phi_{0,t}$ of the triangular FDE \eqref{eqn.triangular.nonlin} is a bijection for any $t\in J$.

\begin{definition}\label{2parameterflow.tri.nonlin}
A two-parameter family of mappings
$$
\phi_{s,t}(\cdot) : \R^d\rightarrow \R^d, \qquad s,t\in J,
$$
 is called  {\em a two-parameter flow in $\R^d$} if the following three conditions are satisfied\\
(i) $\phi_{s,t}(x)$ is continuous as a function of three variables $s,t\in J, x\in\R^d$;\\
 (ii) For any fixed $s,t\in J$ the mapping $\varphi_{s,t}(\cdot)$ is a homeomorphism of $\R^d$;\\
(iii) This family satisfies the following flow property
$$
\phi_{s,t}\circ \phi_{u,s} = \phi_{u,t}\qquad\hbox{for all}\quad u,s,t\in J.
$$
\end{definition} 

\begin{theorem}[Triangular systems of FDEs generate nonlocal dynamical systems]\label{thm.2param.flow.triangular}
 (i) The triangular FDE \eqref{eqn.triangular.nonlin} generates a two-parameter flow  in $\R^d$; namely it generates the two-parameter flow in $\R^d$
 $$
\phi_{s,t} : \R^d\rightarrow \R^d, \qquad s,t\in J,
$$
where $\phi_{0,t}$, $t\in J$, is the evolution mapping of \eqref{eqn.triangular.nonlin} and
$$
\phi_{s,t} := \phi_{0,t}\circ \phi_{0,s}^{-1}\qquad\hbox{for all}\quad s,t\in J.
$$
(ii) The linear triangular system of FDEs \eqref{triangularEq} generates a two-parameter flow of $d$-dimensional linear nonsingular operators.
\end{theorem}
\begin{proof}
Similar to Theorem~\ref{thm.2param.flow}.
\end{proof}

\begin{definition}
The two-parameter flow $\phi_{s,t}$ generated by the triangular FDE \eqref{eqn.triangular.nonlin} is called the {\em nonlocal dynamical system generated by \eqref{eqn.triangular.nonlin}}.
\end{definition}

\section{A general high dimensional system of FDEs does not generate a dynamical system}\label{sec.d-dim}

In this section we show that, in the high dimensional case different trajectories of a FDE can intersect each other. Thus a high dimensional system of FDEs does, in general, not generate a dynamical system. In order for a high dimensional FDE to generate a nonlocal dynamical system we need an additional property of the FDE, for example triangularity.
The results of this sections also provide a counter example to the assertions of Bonilla, Rivero and Trujillo~\cite[Theorem 1, Propositions 1 and 2]{Bonilla2007}.

\begin{theorem}[Different trajectories of a high dimensional system of FDEs can meet]\label{thm.intersection}
For any $d\geq 2$ there exists a system of type \eqref{mainEq} with a property that it has two different solutions $x_1(\cdot), x_2(\cdot)$ with $x_1(0)\not= x_2(0)$ which intersect each other at some finite time moment $0<T<\infty$, i.e., $x_1(T)=x_2(T)$.
\end{theorem}
\begin{proof}
It suffices to construct a two-dimensional system of type \eqref{mainEq} with the desired property.
Below we shall construct a two-dimensional linear autonomous system of FDEs with such a property.

Since $\alpha\in (0,1)$, the complex-valued Mittag-Leffler function $E_\alpha(\cdot)$ has infinitely many zeros in $\C$ (see Gorenflo \textit{et al.}~\cite[Corollary~3.10, p.~30]{GoKiMaRo14}). Take and fix $z^*\in \C$ such that $E_\alpha(z^*)=0$. Let $\phi:= \arg(z^*)\in (-\pi, \pi]$, put $\lambda := \cos\phi + i\sin\phi$, where $i = \sqrt{-1}\in\C$.
Note that since $\alpha\in \R$ we have  $E_\alpha(\overline{z^*}) = \overline{E_\alpha(z^*)} =0$, where $\overline{w}$ denotes the complex conjugate of the complex number $w$. Since $\alpha\in (0,1)$ we have $z^*\notin\R$, hence $\lambda\notin \R$. Let
$$
A:= 
\left(  
\begin{array}{cc}
 \cos\phi    &     \sin\phi \\
  -\sin\phi    & \cos\phi \\ 
   \end{array}
   \right).
$$
Then $A$ has two (complex) eigenvalues $\lambda, \overline\lambda$. We show that the linear FDE
\begin{equation}\label{eqn.2-dim}
    ^{C\!}D_{0+}^\alpha x(t)= Ax(t), \quad t\in\R_+,
\end{equation}
is the desired FDE. Indeed, it is known that this linear FDE is solvable (see Diethelm~\cite[Theorem~7.15, p.~152]{Diethelm2010}) and has two following independent real solutions
$$
 x_1(t):=  \left(  
\begin{array}{c}
 \big(E_\alpha(\lambda t^\alpha) + E_\alpha(\overline\lambda t^\alpha)\big)     
   \\ [6pt]
i\big(E_\alpha(\lambda t^\alpha) - E_\alpha(\overline\lambda t^\alpha)\big)  \\ 
   \end{array}
   \right),\;
  x_2(t):=   \left(  
\begin{array}{c}   
 -i\big(E_\alpha(\lambda t^\alpha) - E_\alpha(\overline\lambda t^\alpha)\big)  \\ [6pt]
 \big(E_\alpha(\lambda t^\alpha) + E_\alpha(\overline\lambda t^\alpha)\big)       \\ 
   \end{array}
   \right).
   $$
Put
$$
u(t) := E_\alpha(\lambda t^\alpha) + E_\alpha(\overline\lambda t^\alpha)
  \qquad
   \hbox{and}\quad
v(t) := i \big(E_\alpha(\lambda t^\alpha) - E_\alpha(\overline\lambda t^\alpha)\big).
$$
Then $u, v : \R_+ \rightarrow \R$, $u(0)=2, v(0)=0$, and
$$
 x_1(t) =  \left(  
\begin{array}{c}
 u(t)    
   \\ 
v(t) \\ 
   \end{array}
   \right),\;\;
  x_2(t) =   \left(  
\begin{array}{c}   
 -v(t) \\ 
 u(t)       \\ 
   \end{array}
   \right),
\quad
   x_1(0) =  \left(  
\begin{array}{c}
 2   
   \\ 
0\\ 
   \end{array}
   \right),\;\;
  x_2(0) =   \left(  
\begin{array}{c}   
 0\\ 
 2     \\ 
   \end{array}
   \right)
   $$
 The general solution of
\eqref{eqn.2-dim} is
\begin{equation}\label{eqn.2dim2}
x(t)= a x_1(t)  + b x_2(t)
\end{equation}
where $a,b\in\R$ are arbitrary real constants. Let $T>0$ be the unique finite positive number satisfying
$$
\lambda T^\alpha = z^*;
$$
such a $T$ exists uniquely due to the definition of $z^*$ and $\lambda$.
Clearly $u(T)=v(T)=0$, hence $x_1(T)=x_2(T)=(0,0)^\rT$. From \eqref{eqn.2dim2} it follows that for any solution $x(t)$ of \eqref{eqn.2-dim} we have $x(T) =0$.
\end{proof}

From Theorem~\ref{thm.intersection} we obtain immediately the following corollary.
\begin{corollary}[A higher dimensional FDE does not generate a nonlocal dynamical system]
For $d\geq 2$ the FDE \eqref{mainEq} does, in general, not generate a two-parameter flow in $\R^d$. Hence it does, in general, not generate a nonlocal dynamical system.
\end{corollary}

\begin{remark}
(i) Actually for the FDE \eqref{eqn.2-dim} {\em all the solutions} are equal $(0,0)^\rT$ at time $T$; hence all they meet each other.\\
(ii) Theorem~\ref{thm.intersection} shows that in contrary to the initial value problem, the terminal value problem for FDEs is not always solvable.\\
(iii) Theorem~\ref{thm.intersection} allows us to better understand the dynamics of the FDEs. It reveals a distinguished feature of FDEs in comparision with ODEs: different trajectories of an ODE cannot meet whereas for FDEs they may meet.\\
(iv) By small modification of $A$ in the proof of Theorem~\ref{thm.intersection} we can make the time of intersection $T$ small.\\
(v) By a small modification of the proof one can show that Theorem~\ref{thm.intersection} also holds for any positive real $\alpha \not=1$.
\end{remark}

\section*{Acknowledgement}
This research of the authors is funded by the Vietnam National Foundation for
Science and Technology Development (NAFOSTED) under Grant Number 101.03-2014.42.

\end{document}